\documentclass[twoside]{article}

\usepackage[accepted]{aistats2020}

\usepackage{amsmath,amsfonts,amssymb,amstext,amsthm}
\usepackage{algpseudocode}
\usepackage{algorithm}
\usepackage{xcolor}
\definecolor{somepurple}{HTML}{663399}
\usepackage[colorlinks=true,linkcolor=blue, citecolor=somepurple]{hyperref}
\usepackage{import}
\usepackage{subcaption}
\usepackage{nicefrac}

\usepackage{arydshln} 

\newcommand{\K}{\mathcal{K}}
\newcommand{\R}{\mathbb{R}}
\newcommand{\dotprod}[2]{\left\langle #1,#2\right\rangle}
\newcommand{\norm}[1]{\left\| #1 \right\|}

\newcommand{\prox}[2]{\mathrm{prox}_{#1}\left(#2\right)}
\newcommand{\n}[1]{\| #1 \|}
\newcommand{\lr}[1]{\langle #1\rangle}

\newtheorem{theorem}{Theorem}
\newtheorem{lemma}{Lemma}
\newtheorem{assumption}{Assumption}
\newtheorem{corollary}{Corollary}

\newcommand{\EE}{\mathbb{E}}

\newcommand{\cO}{{\cal O}}
\newcommand{\cX}{{\cal X}}
\newcommand{\cY}{{\cal Y}}

\newcommand{\eqdef}{\stackrel{\text{def}}{=}}
\newcommand{\mA}{{\bf A}}
\newcommand{\mB}{{\bf B}}

\newcommand{\mI}{{\bf I}}

\newcommand{\mT}{{\bf T}}
\newcommand{\mU}{{\bf U}}
\newcommand{\mV}{{\bf V}}

\newcommand{\mSigma}{{\bf \Sigma}}
\def\<#1,#2>{\left\langle #1,#2\right\rangle}
\usepackage[colorinlistoftodos,bordercolor=orange,backgroundcolor=orange!20,linecolor=orange,textsize=scriptsize]{todonotes}

\usepackage[round]{natbib}

\begin{document}
\twocolumn[

\aistatstitle{Revisiting Stochastic Extragradient}

\aistatsauthor{ Konstantin Mishchenko\textsuperscript{1} \And Dmitry Kovalev\textsuperscript{1} \And Egor Shulgin\textsuperscript{1,~3} \AND Peter Richt\'{a}rik\textsuperscript{1} \And Yura Malitsky\textsuperscript{2}}
\runningauthor{Konstantin Mishchenko, Dmitry Kovalev, Egor Shulgin, Peter Richt\'{a}rik, Yura Malitsky}

\aistatsaddress{\textsuperscript{1}KAUST \And  \textsuperscript{2}EPFL \And  \textsuperscript{3}MIPT}]
	
\begin{abstract}
We fix a fundamental issue in the stochastic extragradient method by providing a new sampling strategy that is motivated by approximating implicit updates. Since the existing stochastic extragradient algorithm, called Mirror-Prox, of~\citep{juditsky2011solving} diverges on a simple bilinear problem when the domain is not bounded, we prove guarantees for solving variational inequality that go beyond existing settings. Furthermore, we illustrate numerically that the proposed variant converges faster than many other methods on bilinear saddle-point problems. We also discuss how extragradient can be applied to training Generative Adversarial Networks (GANs) and how it compares to other methods. Our experiments on GANs demonstrate that the introduced approach may make the training faster in terms of data passes, while its higher iteration complexity makes the advantage smaller.
\end{abstract}
\section{Introduction}
Algorithmic machine learning has for a long time been centered around minimization of a single function. A lot of works are still targeting solving empirical risk minimization and new results touch upon methods as old as gradient descent itself.  

However, as the gap between lower bounds and available minimization algorithms is shrinking, the focus is shifting towards more challenging problems such as variational inequality, where a significant number of long-unresolved questions is remaining. This problem has a rich history with applications in economics and computer science, but the arising applications provide new desiderata on algorithm properties. In particular, due to high dimensionality and large scale of the corresponding problems, we shall consider the impact of having a {\em stochastic} objective. In particular, recently invented generative adversarial neural networks~\citep{goodfellow2014generative} are often trained using schemes that resemble primal-dual and variational inequality methods, which we shall discuss in detail later.

Variational inequality can be seen as an extension of the necessary first-order optimality condition for minimization
problem, which is also sufficient in the convex case.
When the operator involved in its formulation is monotone and is equal to
the gradient of a function, this corresponds to convex
minimization.

Formally, the problem that we consider is that of finding a point $x^*$ satisfying
\begin{equation}
	g(x) - g(x^*)+ \dotprod{F(x^*)}{x - x^*} \geq 0, \text{ for all } x \in \R^d,\label{eq:vi}
\end{equation}
where
$g\colon \R^d\to {\R}\cup \{+\infty\} $ is a proper lower
semi-continuous convex function and 
$F\colon \R^d \rightarrow \R^d$ is a monotone operator. Some application of interest are not covered by the monotonicity framework, but, unfortunately, little
is known about variational inequality and even saddle point problems when
monotonicity is missing. Thus, we stick to this assumption and rather try
to model oscillations arising in some problems by considering particularly
unstable~\citep{gidel2018negative, chavdarova2019reducing} bilinear
minimax problems.

Of particular interest to us is the situation where $F(x)$ is the expectation with respect to random variable $\xi$ of the random operator $F(x;\xi)$. This formulation has two aspects. First, one can model data distribution, especially when a large dataset is available and the problem is that of minimizing empirical loss. Second, $\xi$ can be a random variable sampled by one of the GAN networks, called \textit{generator}. In any case, throughout the work we assume that we sample unbiased estimates $F(\cdot; \xi)$ of $F(\cdot)$ such that $\EE_\xi F(\cdot; \xi) = F(\cdot)$.

Let us explicitly mention that a special case of~\eqref{eq:vi} is constrained saddle point optimization,
\begin{align*}
    \min_{x\in\cX} \max_{y\in \cY} f(x, y),
\end{align*}
where $\cX$ and $\cY$ are some convex sets and $f$ is a smooth function. While this example looks deceptively simple, simultaneous gradient descent-ascent is known to diverge on this problem~\citep{goodfellow2016nips} even when $f$ is convex-concave. In particular, the objective $f(x, y)=x^\top y$ leads to geometrical divergence for any nontrivial initialization~\citep{daskalakis2018training}. See~\citep{mishchenko2019stochastic} for more applications of the convex-concave saddle point problem in machine learning and~\citep{gidel2018variational} for extra discussion on variational inequality and its relation to GANs.

\subsection{Related work}
The extragradient method was first proposed by~\citep{korpelevich1977extragradient}. Since then there have been developed a number of its extensions, most famous of which is the Mirror-Prox method~\citep{nemirovski2004prox} that uses mirror descent update. At each iteration, the standard extragradient method is trying to approximate the implicit update, which is known to be much more stable. Assuming the operator is Lipschitz, it is enough to compute the operator twice to do the approximation accurately enough. We base our intuition upon this property and we shall discuss it in detail later in the paper.

While extragradient uses future information, i.e., information from one gradient step ahead, past information can also help to stabilize convergence. In particular, \textit{Optimistic mirror descent} (OMD), first proposed by~\citep{rakhlin2013online} for convex-concave zero-sum games, has been analyzed in a number of works~\citep{mokhtari2019unified, daskalakis2018last, gidel2018variational} and it was applied to GAN training in~\citep{daskalakis2018training}. The rates that we prove in this work for stochastic extragradient match the best known results for OMD, but are given under more general assumptions. Moreover, the method of~\citep{gidel2018variational} diverges on bilinear problems.

Many other techniques also allow to improve stability and achieve convergence for monotone operators in the particular case of saddle point problems. For instance, \textit{alternating} gradient descent-ascent does not, in general,  converge to a solution~\citep{gidel2018negative}, the negative momentum trick proposed in~\citep{gidel2018negative} can fix this. 

We note that our work is not the first to consider a variant of stochastic extragradient. A stochastic version of the Mirror-Prox method~\citep{nemirovski2004prox} was analyzed in~\citep{juditsky2011solving} under pretty restrictive assumptions. While deterministic extragradient approximates implicit update, the authors of~\citep{juditsky2011solving} chose to sample two different instances of the stochastic operator, which leads to a poor approximation of stochastic implicit update unless the variance is tiny. It was observed in~\citep{chavdarova2019reducing} that this approach leads to terrible practical performance, dubious convergence guarantees and divergence on bilinear problems. All later variants of stochastic extragradient, that we are aware of, consider the same update model.

Surprisingly, a variant of extragradient was also rediscovered by practitioners~\citep{metz2016unrolled} as a way to stabilize training of GANs. The main difference of the method of~\citep{metz2016unrolled} to what we consider is in applying extra steps only on one of two neural networks. In addition, \citep{metz2016unrolled} proposed to use more than one extra step and claim that in on specific problems 5 steps is a good trade-off between results quality and computation.

\citep{chavdarova2019reducing} showed that the methods of~\citep{juditsky2011solving} and~\citep{gidel2018variational} diverge on stochastic bilinear saddle point problem. As a fix, they proposed a stochastic extragradient method with variance reduction (SVRE), which achieves a linear rate $\cO ((n + \frac{L}{\mu})\log \frac{1}{\varepsilon})$. However, their theory works only for saddle point problems and it does not cover the case without strong monotonicity, so it is less general than ours.

\subsection{Theoretical background}
Here we provide several technical assumptions that are standard for variational inequality.
\begin{assumption}
	Operator $F\colon \R^d\to \R^d$ is monotone, that is 
	$
		\dotprod{F(x) - F(y)}{x-y} \geq 0
	$ for all $x,y \in \R^d$.
	In stochastic case, we assume that $F(x; \xi)$ is monotone almost surely.
\end{assumption}
The monotonicity assumption is an extension of the notion of convexity
and is quite standard in the literature. There are several
versions of pseudo-monotonicity, but without it the variational
inequality problem becomes extremely hard to solve. 

\begin{assumption}
	Operator $F(\cdot; \xi)$ is almost-surely $L$-Lipschitz,  that is for all $x,y \in \R^d$
	\begin{equation}
		\norm{F(x;\xi) - F(y;\xi)} \leq L\norm{x-y}.
	\end{equation}
\end{assumption}
In addition to operator monotonicity, we ask for convexity and some regularity properties of $g(\cdot)$ as given below.
\begin{assumption}
	Function $g\colon \R^d \rightarrow {\R}\cup \{+\infty\}$ is
        lower semi-continuous and $\mu$-strongly convex for $\mu\geq 0$, i.e., for all $x, y\in \R^d$ and any $h\in \partial g(y)$
	\begin{align*}
		g(x) - g(y) - \<h, x - y>
		\ge \frac{\mu}{2}\|x - y\|^2.
	\end{align*}
	If $\mu=0$, then $g$ is just convex.
\end{assumption}
Even in simple minimization problems, the classical theoretical analysis of stochastic methods ask for uniformly bounded variance, an assumption rarely satisfied in practice. Recent developments of the theory for SGD have removed this assumption, but we are not aware of any results in more general settings. Thus, it is one of our contributions is to relax the uniform variance bound the one below.
\begin{assumption}
	In the strongly convex case, we assume that $F$ has bounded variance at the optimum, i.e.,
	\begin{align*}
		\EE \|F(x^*; \xi) - F(x^*)\|^2 
		\le \sigma^2.
	\end{align*}
\end{assumption}
Depending on the assumptions, we will either work with the variance at the optimum or with a merit function, which involves the variance of a bounded set.

\begin{algorithm}[t]
\caption{Same-Sample Stochastic Extragradient Method for Variational Inequality.}
\label{alg:eg_vi}
\begin{algorithmic}[1]
	\State {\bf Parameters:} $x^0 \in \K$, stepsize $\eta > 0$ 
	\For{$t = 0,1,2,\ldots$}
	\State Sample $\xi^t$
	\State $y^t = \prox{\eta g}{x^t - \eta F(x^t; \xi^t)}$
	\State $x^{t+1} = \prox{\eta g}{x^t - \eta F(y^t; \xi^t)}$
	\EndFor
\end{algorithmic}	
\end{algorithm}
\begin{algorithm}[t]
   \caption{The extragradient method for min-max problems.}
   \label{alg:eg_minimax}
\begin{algorithmic}[1]
   \Require Stepsizes $\eta_1, \eta_2$, initial vectors $x^0$, $y^0$
   \For{$t=0,1,\dotsc$}
	   \State $u^t = x^t - \eta_1 \nabla_x f(x^t, y^t)$
	   \State $v^t = y^t + \eta_1 \nabla_y f(x^t, y^t)$
	   \State $x^{t+1} = x^t - \eta_2 \nabla_x f(u^t, v^t)$
	   \State $y^{t+1} = y^t + \eta_2 \nabla_y f(u^t, v^t)$
   \EndFor
\end{algorithmic}
\end{algorithm}

\section{Theory}
It is known that implicit updates are more stable when solving
variational inequality and sometimes it is argued that the main goal
of algorithmic design is to approximate those~\citep{mokhtari2019unified}. From that
perspective, the current stochastic extragradient, which was suggested
in~\citep{juditsky2011solving}, does not make much sense. Since it uses
two independent samples, it will rarely approximate the implicit
update, so it is rather not surprisingly that it fails on bilinear
problems.

To better explain this phenomenon, below we show that extragradient efficiently approximates implicit update.
\begin{theorem}\label{th:approx}
	Let $F$ be an $L$-Lipschitz operator and define $y\eqdef \prox{\eta g}{x - \eta F(x)}$, $z\eqdef \prox{\eta g}{x - \eta F(y)}$, $w\eqdef \prox{\eta g}{x - \eta F(w)}$, where $\eta>0$ is any stepsize. Then,
	\begin{align*}
		\|w - z\|
		\le \eta^2 L^2 \|w - x\|.
	\end{align*}
\end{theorem}
The right-hand side in Theorem~\ref{th:approx} serves as a measure of stationarity and decreases as $x$ gets closer to the problem's solution. The essential part of the bound is that the error is of order $O(\eta^2)$ rather than $O(\eta)$. This allows the approximation to be better than simple gradient update making it possible for the method to solve variational inequality. One can also mention that having extra factor of $\eta L$ is beneficial only when $\eta<\nicefrac{1}{L}$, which provides a good intuition on why extragradient uses smaller stepsizes than gradient.

However, when the stochastic update is used, this result is not applicable directly. If two different samples of the operator are used, $F(\cdot; \xi^t)$ and $F(\cdot; \xi^{t+1/2})$, as is done in stochastic Mirror-Prox~\citep{juditsky2011solving},  then the update does not seem to approximate implicit update of any operator. This is why we propose in this work to use the same sample, $\xi^t$, when computing $y^t$ and $x^{t+1}$, see Algorithm~\ref{alg:eg_vi}. Equipped with our update, we are always approximating the implicit update of stochastic operator $F(\cdot; \xi^t)$ and our theoretical results suggest that this is the right approach.
\subsection{Stochastic variational inequality}
Our first goal is to show that our stochastic version of the extragradient method converges for strongly monotone variational inequality. The next theorem provides the rate that we obtained.
\begin{theorem}\label{th:linear_rate_str_mon}
	Assume that $g$ is a $\mu$-strongly convex function, operator
        $F(\cdot; \xi)$ is almost surely monotone and $L$-Lipschitz,
        and that its variance at the optimum $x^*$ is bounded by constant, $\EE \|F(x^*; \xi) - F(x^*)\|^2 \le \sigma^2$. Then, for any $\eta\le \nicefrac{1}{(2L)}$
	\begin{align*}
		\EE \|x^t - x^*\|^2
		\le \left(1 - \nicefrac{2\eta\mu}{3}\right)^t\|x^0 - x^*\|^2 + \nicefrac{3\eta\sigma^2}{\mu}.
	\end{align*}
\end{theorem}
In the case where at the optimum the noise is zero, we recover a slight generalization of linear convergence of extragradient~\citep{tseng1995linear}. This is also similar to the rate proved for optimistic mirror descent in~\citep{gidel2018variational}, however we do not ask for uniform bounds on the variance. Therefore, we believe that this result is significantly more general.
 
\begin{theorem}\label{th:ergodic_conv}
	Let $g$ be a convex function, $F(\cdot; \xi)$ be monotone and
        $L$-Lipschitz almost surely. Then, the iterates of
        Algorithm~\ref{alg:eg_vi} with stepsize
        $\eta=\cO(\nicefrac{1}{(\sqrt{t}L)})$ satisfy for any 
        set $\cX$ and $x\in \cX$
	\begin{align*}
		\EE \left[g(\hat x^t) - g(x) + \< F(x), \hat x^t - x> \right]\\
		\le \frac{1}{\sqrt{t}L}\sup_{x\in\cX} \left\{\frac{L^2}{2}\|x^0 - x\|^2+ \sigma_x^2\right\}.
	\end{align*}
	where $\hat x^t = \frac{1}{t}\sum_{k=0}^t y^k$ and
        $\sigma_x^2\eqdef \EE \|F(x) - F(x;\xi)\|^2$, i.e., $\sigma_x^2$ is the
        variance of $F$ at point $x$.
\end{theorem}
The left-hand side in the bound above is a merit function that has been used in variational inequality literature~\citep{nesterov2007dual}. This result is more general than the one obtained in~\citep{gidel2018variational}, where the authors require for the same rate bounded variance and even $\EE\|F(x; \xi)\|^2\le M<\infty$  uniformly over $x$.

In fact, the claim that we prove in the appendix is a bit more general than the one presented in the previous theorem. If we know that $\sigma_x$ is sufficiently small on a bounded set $\cX$, then we can get a $\cO(\nicefrac{1}{t} + \sup_{x\in \cX} \sigma_x)$ rate, i.e., fast convergence to a neighborhood.
\subsection{Adversarial bilinear problems}
The work~\citep{gidel2018negative} argues that a good illustration of method's stability can be obtained when considering minimax bilinear problems, which is given by
\begin{align*}
    \min_x \max_y f(x, y) = x^\top \mB y + a^\top x + b^\top y,
\end{align*}
where $\mB$ is a full rank square matrix. One can show that if there exists a Nash equilibrium point, then $f(x, y) = (x - x^*)^\top \mB (y - y^*) + \textrm{const}$ for some pair $(x^*, y^*)$\footnote{If $a$ does not belong to the column space of $B$ or $b$ does not belong to the column space of $B^\top$, the unconstrained minimax problem admits no equilibrium. Otherwise, if we introduce $\tilde a, \tilde b$ such that $a=-B y^*$ and $b=-B^\top x^*$, we have $(x - x^*)^\top B(y-y^*) = x^\top By + a^\top x + b^\top y + (x^*)^\top B y^* $.}. This problem is particularly interesting because simple gradient descent-ascent diverges geometrically when solving it, 
\begin{theorem}\label{th:eg_bilinear}
    Let $f$ be bilinear with a full-rank matrix $\mB$ and apply Algorithm~\ref{alg:eg_minimax} to it. Choose any $\eta_1$ and $\eta_2$ such that $\eta_2<\nicefrac{1}{\sigma_{\max}(\mB)}$ and  $\eta_1\eta_2 < \nicefrac{2}{\sigma_{\max}(\mB)^2}$, then the rate is
    \begin{align*}
        \|x^t - x^*\|^2 + \|y^t - y^*\|^2
        \le \rho^{2t}(\|x^0 - x^*\|^2 + \|y^0 - y^*\|^2),
    \end{align*}
    where $\rho\eqdef  \max\left\{(1 -
        \eta_1\eta_2\sigma_{\max}(\mB)^2)^2 +
        \eta_2^2\sigma_{\max}(\mB)^2\right.$, $ \left.(1 -
        \eta_1\eta_2\sigma_{\min}(\mB)^2)^2 +
        \eta_2^2\sigma_{\min}(\mB)^2\right\}$.
\end{theorem}
The conditions for $\eta_1$ and $\eta_2$ in Theorem~\ref{th:eg_bilinear} are necessary, but not sufficient. To guarantee convergence, one needs to have $\rho<1$ and below we provide two such examples.
\begin{corollary}\label{cor:bilinear}
    Under the same assumption as in Theorem~\ref{th:eg_bilinear}, consider two choices of stepsizes:
    \begin{enumerate}
    		\item if $\eta_1=\eta_2 = \nicefrac{1}{(\sqrt{2}\sigma_{\max}(\mB))}$ we get
    \begin{align*}
        &\|x^t - x^*\|^2 + \|y^t - y^*\|^2\\
        &\le \left(1 - \frac{\sigma_{\min}(\mB)^2}{6\sigma_{\max}(\mB)^2} \right)^{2t}(\|x^0 - x^*\|^2 + \|y^0 - y^*\|^2),
    \end{align*}
    		\item if $\sigma_{\min}(\mB) > 0$, and $\eta_1 = \nicefrac{\kappa}{(\sqrt{2}\sigma_{\max}(\mB)^2)}$, $\eta_2 = \nicefrac{1}{(\sqrt{2}\kappa \sigma_{\max}(\mB)^2)}$ with $\kappa \eqdef \nicefrac{\sigma_{\min(\mB)}^2}{\sigma_{\max(\mB)}^2}$, then the rate is
    \begin{align*}
        &\|x^t - x^*\|^2 + \|y^t - y^*\|^2 \\
        &\le \left(1 - \frac{\sigma_{\min}(\mB)^2}{4\sigma_{\max}(\mB)^2} \right)^{2t}(\|x^0 - x^*\|^2 + \|y^0 - y^*\|^2).
    \end{align*}
	\end{enumerate}    
\end{corollary}
If we denote $\kappa \eqdef \frac{\sigma_{\min(\mB)}^2}{\sigma_{\max(\mB)}^2}$ as in~\citep{mokhtari2019unified}, then the complexity in both cases is $O(\kappa \log\frac{1}{\varepsilon})$. However, we provide this result for potentially different stepsizes to obtain new insights about how they should be chosen. One can see, in particular, that choosing a huge $\eta_1$ is possible if $\eta_2$ is chosen small, but not vice versa.

\section{Nonconvex extragradient}
Since the objective of neural networks is not convex, it is desirable to have a
guarantee for convergence that would not assume operator
monotonicity. Alas, there is almost no theory even for nonconvex
minimax problems and full gradient updates as even the notion of
stationarity becomes tricky. Therefore, in this section we only
discuss the method performance when minimizing loss function.

\begin{figure*}[t]
    \centering
      \begin{minipage}[b]{0.49\textwidth}
    \includegraphics[width=1\linewidth]{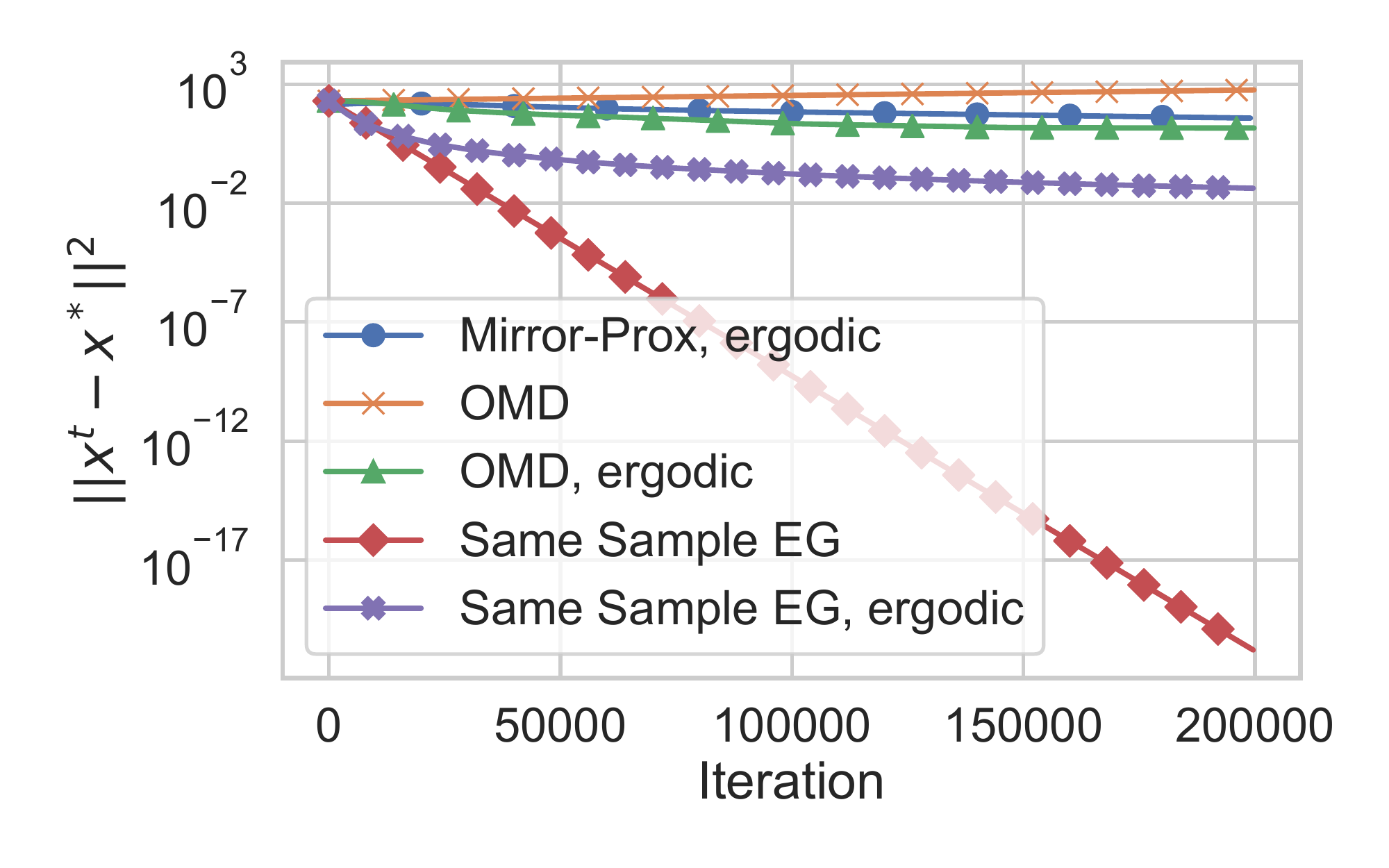}
\end{minipage}
\begin{minipage}[b]{0.49\textwidth}
    \includegraphics[width=1\linewidth]{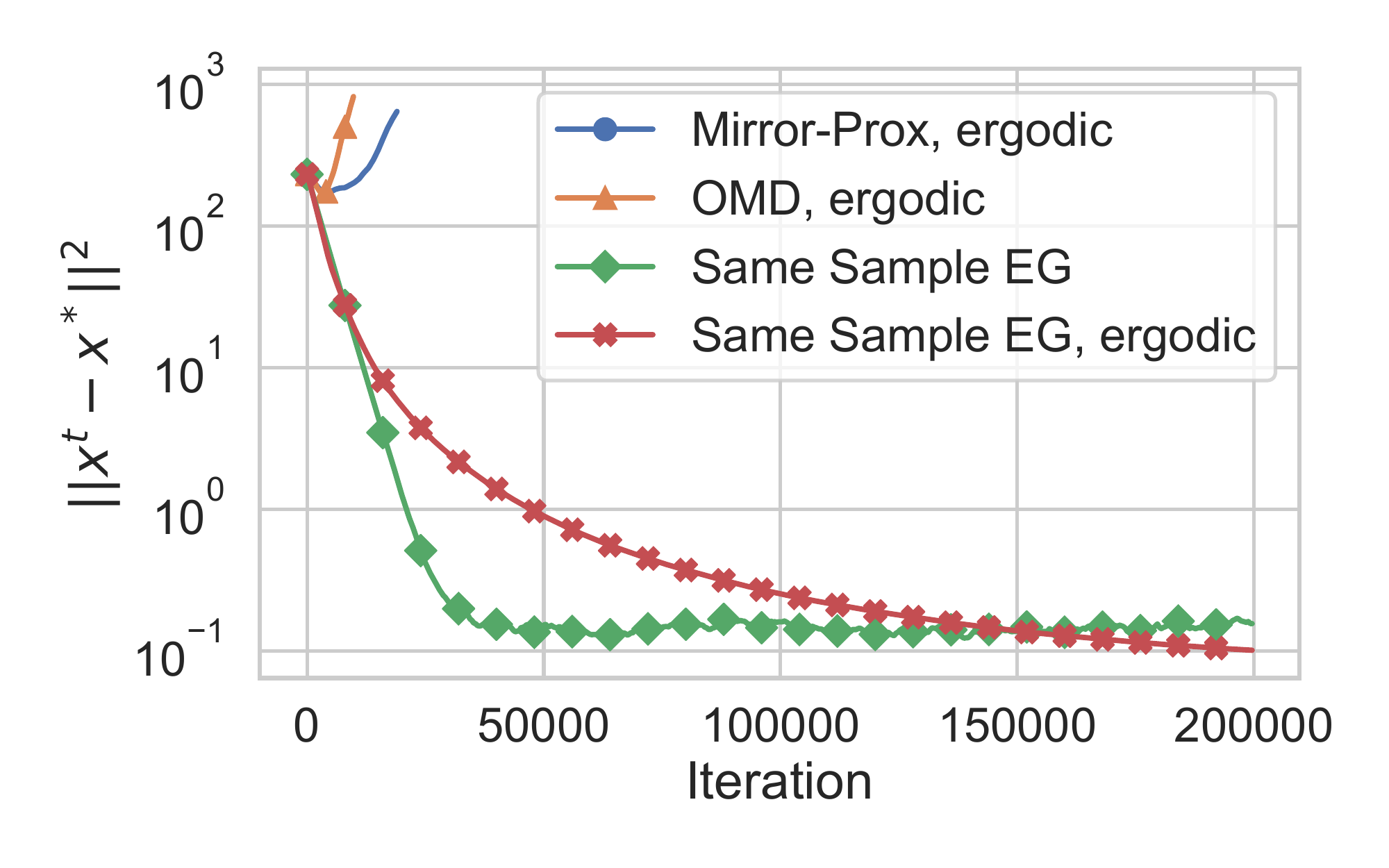}
\end{minipage}
\caption{Left: comparison of using independent samples and averaging as suggested by~\citep{juditsky2011solving} and the same sample as proposed in this work. The problem here is the sum of randomly sampled matrices $\min_x\max_y\sum_{i=1}^nx^\top \mB_i y$. Since at point $(x^*, y^*)$ the noise is equal $0$, the convergence of Algorithm~\ref{alg:eg_vi} is linear unlike the slow rates of~\citep{juditsky2011solving} and~\citep{gidel2018variational}. 'EGm' is the version with negative momentum~\citep{gidel2018negative} equal $\beta=-0.3$. Right: bilinear example with linear terms.}
\label{fig:bilinear}
\end{figure*}

\begin{figure*}[t]
	\centering
	\includegraphics[scale=0.09]{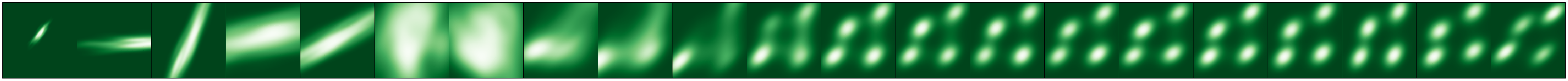}
	\includegraphics[scale=0.09]{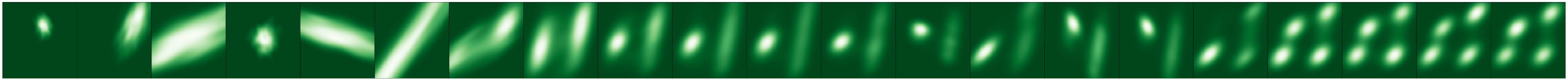}
	\includegraphics[scale=0.09]{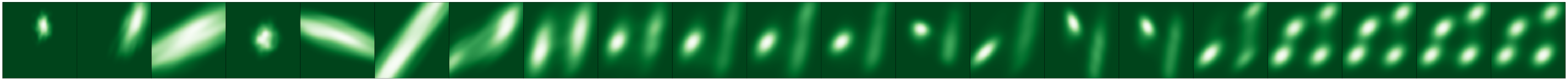}
	\caption{Top line: extragradient with the same sample. Middle line: gradient descent-ascent. Bottom line: extragradient with different samples. Since the same seed was used for all methods, the former two methods performed extremely similarly, although when zooming it should be clear that their results are slightly different.}
	\label{fig:four_gaussians}
\end{figure*}

Formally, the problem that we consider here is
\begin{align}\label{eq:erm}
	\min_{x} \EE_\xi f(x; \xi),
\end{align}
where $f$ is a smooth bounded from below and potentially nonconvex
function. To show convergence, we need the following standard
assumption.
\begin{assumption}
	There exists a constant $\sigma > 0$ such that for all $x$ it holds
	\begin{align*}
		\EE \|\nabla f(x; \xi) - \nabla f(x)\|^2
		\le \sigma^2.
	\end{align*}
\end{assumption}
Then, we are able to show that the method converges to a local
minimum.
\begin{theorem}\label{th:nonconvex}
	Choose $\eta\le \frac{1}{4L}$ and apply extragradient to~\eqref{eq:erm}. Then, its iterates satisfy
	\begin{align*}
		\EE\|\nabla f(\hat x^t)\|^2
		\le \frac{5}{\eta t}(f(x^{0}) - f^*) + 11\eta L\sigma^2,
	\end{align*}
	where $\hat x^t$ is sampled uniformly from $\{x^0, \dotsc,
        x^{t-1}\}$ and $f^* = \inf_x f(x)$.
    \end{theorem}
\begin{corollary}
	If  we choose $\eta = \Theta\left(\nicefrac{1}{(L\sqrt{t})} \right)$, then the rate is $O\left(\nicefrac{(f(x^0) - f^*)}{\sqrt{t}} + \nicefrac{\sigma^2}{\sqrt{t}}\right)$, which is the same as the rate of SGD under our assumptions.
\end{corollary}
The statement of the theorem almost coincides with that of SGD, see for instance~\citep{ghadimi2013stochastic}. This suggests that extragradient in most cases should not be seen as an alternative to SGD. We also provide a simple experiment with training Resnet-18~\citep{he2016deep} on Cifar10~\citep{krizhevsky2009learning} in Appendix~\ref{sec:erm_exp}, which gives a similar message.

\section{Experiments}

\subsection{Bilinear minimax}
In this experiment, we generated a matrix with entries from standard normal distribution and dimensions 200. Since we did not observe much difference when changing the matrix size, we provide only one run in Figure~\ref{fig:bilinear}. The results are very encouraging and show the superiority of the proposed approach on this problem. We provide two cases, with zero noise at the optimum and non-zero noise. In the latter case, only our method did not diverge. 

When the noise at the optimum is zero, this is mostly like deterministic case for our method, but for the rest it is a difficult problem. On the other hand, when the noise is not equal 0 at the solution, the ergodic convergence of our method is faster, just as predicted by Theorem~\ref{th:ergodic_conv}.

\begin{figure*}[t]
\centering
{\includegraphics[width=1\linewidth]{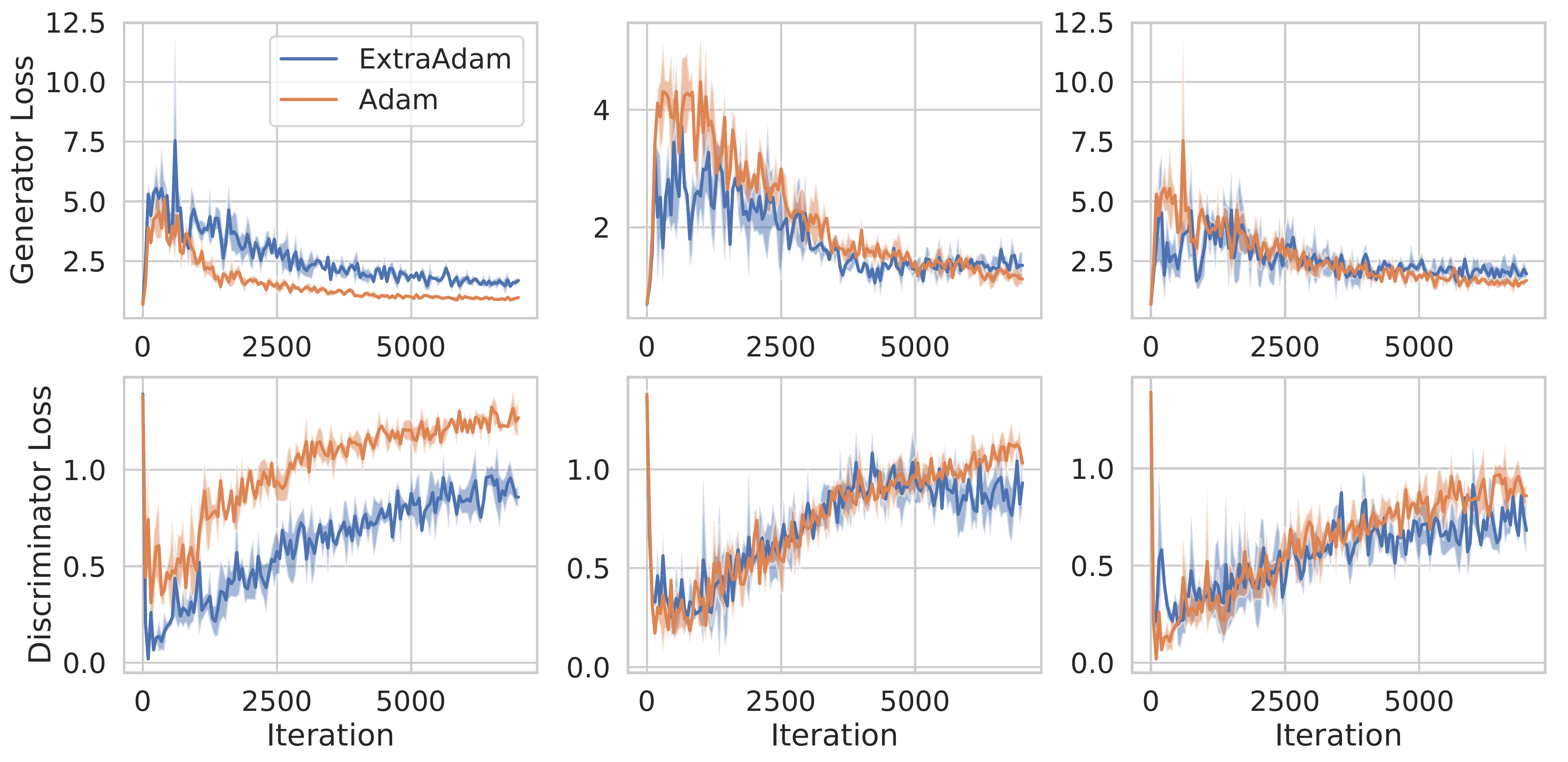}}
 \caption{The columns differ in step sizes for generator and discriminator: 1) ($10^{-4}$, $10^{-4}$), 2) ($5\cdot 10^{-5}$, $5\cdot 10^{-5}$), 3) ($10^{-4}$, $5\cdot 10^{-5}$).
  In the top row, we show the generator loss and in the bottom row that of discriminator.} \label{fig:Adam_ExtraAdam}
\end{figure*}

\subsection{Generating mixture of Gaussians}
Here we compare gradient descent-ascent as well as Mirror-Prox to our method on the task of learning mixture of 4 Gaussians. We provide the evolution of the process in Figure~\ref{fig:four_gaussians}, although we note that the process is rather unstable and all results should be taken with a grain of salt.

To our surprise, negative momentum was rarely helpful and even positive momentum sometimes was giving significant improvement. We suspect that this is due to the different roles of generator and discriminator, but leave further exploration for future work. 

The details of the experiment are as follows. For generator we use neural net with 2 hidden layers of size 16 and tanh activation function and output layer with size 2 and no activation function, which represents coordinates in 2D. Generator uses standard Gaussian vector of size 16 as an input.
For discriminator we use neural net with input layer of size 2, which takes a point from 2D, 2 hidden layers of size 16 and tanh activation function and output layer with size 1 and sigmoid activation function, which represents probability of input point to be sampled from data distribution.
We choose the same stepsize $5\cdot 10^{-3}$ for all methods, which is close to maximal possible stepsize under which the methods rarely diverge.

\begin{table*}[!htb]
    \begin{minipage}{.5\linewidth}
      \centering
        \begin{tabular}{c}
        \hline
        \textbf{Generator}                                     \\ \hline
        \textit{Input}: $z \in \mathbb{R}^{100} \sim \mathcal{N}(0, I)$ \\ \hdashline
        Embedding layer for the label                          \\
        Linear (110 $\to$ 256)                                 \\
        LeakyReLU (negative slope: 0.2)                        \\
        Linear (256 $\to$ 512)                                 \\
        LeakyReLU (negative slope: 0.2)                        \\
        Linear (512 $\to$ 1024)                                \\
        LeakyReLU (negative slope: 0.2)                        \\
        Linear (1024 $\to$ 784)                                \\
        \textit{Tanh}($\cdot$)                                  \\ \hline
        \end{tabular}
    \end{minipage}%
    \begin{minipage}{.5\linewidth}
      \centering
        \begin{tabular}{c}
        \hline
        \textbf{Discriminator}                                 \\ \hline
        \textit{Input}: $x \in \mathbb{R}^{1 \times 28 \times 28} $ \\ \hdashline
        Embedding layer for the label                          \\
        Linear (794 $\to$ 1024)                                \\
        LeakyReLU (negative slope: 0.2)                        \\
        Dropout ($p$=0.3)                                      \\
        Linear (1024 $\to$ 512)                                \\
        LeakyReLU (negative slope: 0.2)                        \\
        Dropout ($p$=0.3)                                      \\
        Linear (512 $\to$ 256)                                 \\
        LeakyReLU (negative slope: 0.2)                        \\
        Dropout ($p$=0.3)                                      \\
        Linear (1024 $\to$ 784)                                \\
        \textit{Sigmoid}($\cdot$)                              \\ \hline
        \end{tabular}
    \end{minipage} 
    \caption{Architectures used for our experiments on \emph{Fashion MNIST}.}
    \label{architectures}
\end{table*}

\subsection{Comparison of Adam and ExtraAdam}

Unfortunately, pure extragradient did not perform extremely well on big datasets, so for the Fashion MNIST and Celeba experiments we used adaptive stepsizes as in Adam~\citep{kingma2014adam}.

In the first set of experiments, we compared the performance of ExtraAdam~\citep{gidel2018variational} and Adam in a Conditional GAN~\citep{mirza2016conditionalgan} setup on Fashion MNIST~\citep{xiao2017fashion} dataset. The generator and discriminator were simple feedforward networks (detailed architectures description in Table~\ref{architectures}). Optimizers were run with mini-batch size of 64 samples, no weight decay and $\beta_1=0.5, \beta_2=0.999$. One iteration of ExtraAdam was counted as two due to a double gradient calculation.
The results (mean and variance) are depicted in Figure~\ref{fig:Adam_ExtraAdam} and were obtained using 3 runs with different seeds.
One can see that extragradient is slower because of the need to compute twice more gradients.

We suspect that Adam is faster partially due to that the problem's structure is something more specific than just a variational inequality. One validation of this guess is that in~\citep{gidel2018negative}, the networks were trained with negative momentum only on discriminator, while generator was trained with constant momentum $+0.5$. Another reason we make this conjecture is that in~\citep{metz2016unrolled} there was proposed a method that can be seen as a variant of extragradient, in which parameters of only one network requires extra steps.

In the second experiment, following~\citep{chavdarova2019reducing}, we trained Self Attention GAN~\citep{zhang2018self}. We note that the loss was generally an ambiguous metric of method comparison, so we provide the Inception score~\citep{salimans2016improved}\footnote[1]{We used implementation from \href{https://github.com/sbarratt/inception-score-pytorch}{this} GitHub repository.} in Figure~\ref{is_celeba} as performance measure for image synthesis. Besides, samples generated after training for two epochs are provided in Figure~\ref{fig:celeba} in the Appendix. 

The work~\citep{gidel2018negative} suggests using negative momentum to
improve game dynamics and achieve faster convergence of the
iterates. We consider using two types of momentum together: $\beta_1$
in the first step and $\beta_2$ in the second, i.e., we use $y^t = x^t - \eta_1 F(x^t; \xi^t) + \beta_1(x^t - x^{t-1})$ and $x^{t+1} = x^t - \eta_2 F(y^t; \xi^t) + \beta_2(x^t - x^{t-1})$.
 Detailed investigation
on bilinear problems shows that $\beta_1$ can be chosen to be positive
and $\beta_2$ should rather be negative. Intuitively, positive
$\beta_1$ allows the method to look further ahead, while negative
$\beta_2$ compensates for inaccuracy in the approximation of implicit
update. In Appendix~\ref{sec:momentum}, we discuss it in more
details. 

The results (mean and variance) are depicted in Figure~\ref{fig:Adam_ExtraAdam} and were obtained using 3 runs with different seeds.

\begin{figure}[H]
\centering
{\includegraphics[width=0.9\linewidth]{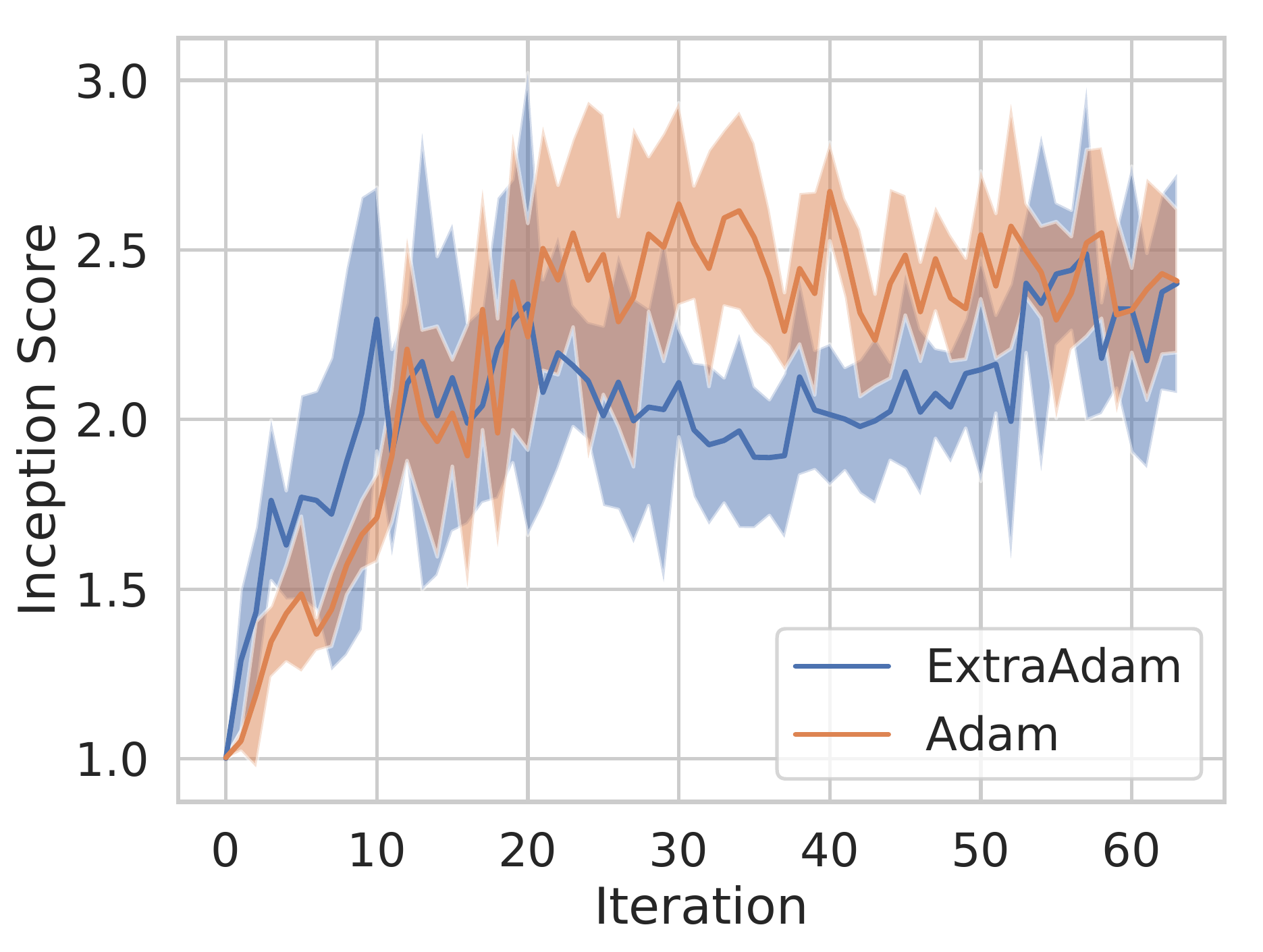}}
\caption{Inception score (mean and variance obtained by 5 runs) computed every 50 iterations during the training process on \emph{CelebA} dataset for 2 epochs.}
\label{is_celeba}
\end{figure}

\begin{figure*}[h]
\centering
\begin{subfigure}[t]{0.43\textwidth}
{\includegraphics[trim={3mm 0 3mm 0},clip,width=1\textwidth]{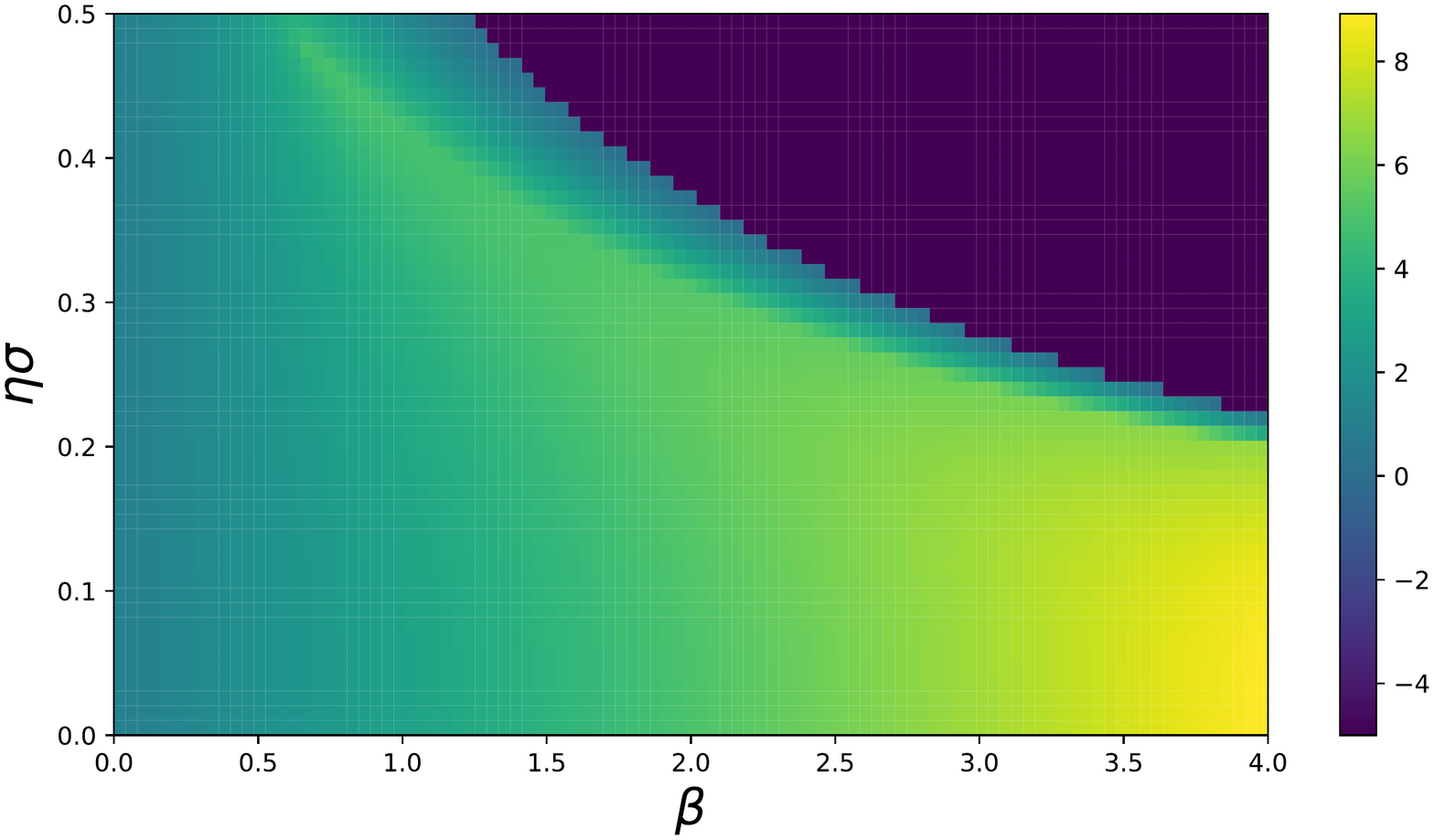}}
\caption{$\eta_1=\eta_2$, $\beta_2=0$, $\beta=\beta_1$ is the $x$-axis, $\eta\sigma_i$ is the $y$-axis. The optimal value of $\beta_1$ depends on $\eta\sigma_i$ and only for small values is significantly bigger 0. The dark area is where the method diverges.\\}
\end{subfigure}
\hfill
\begin{subfigure}[t]{0.43\textwidth}
{\includegraphics[trim={3mm 0 3mm 0},clip,width=1\textwidth]{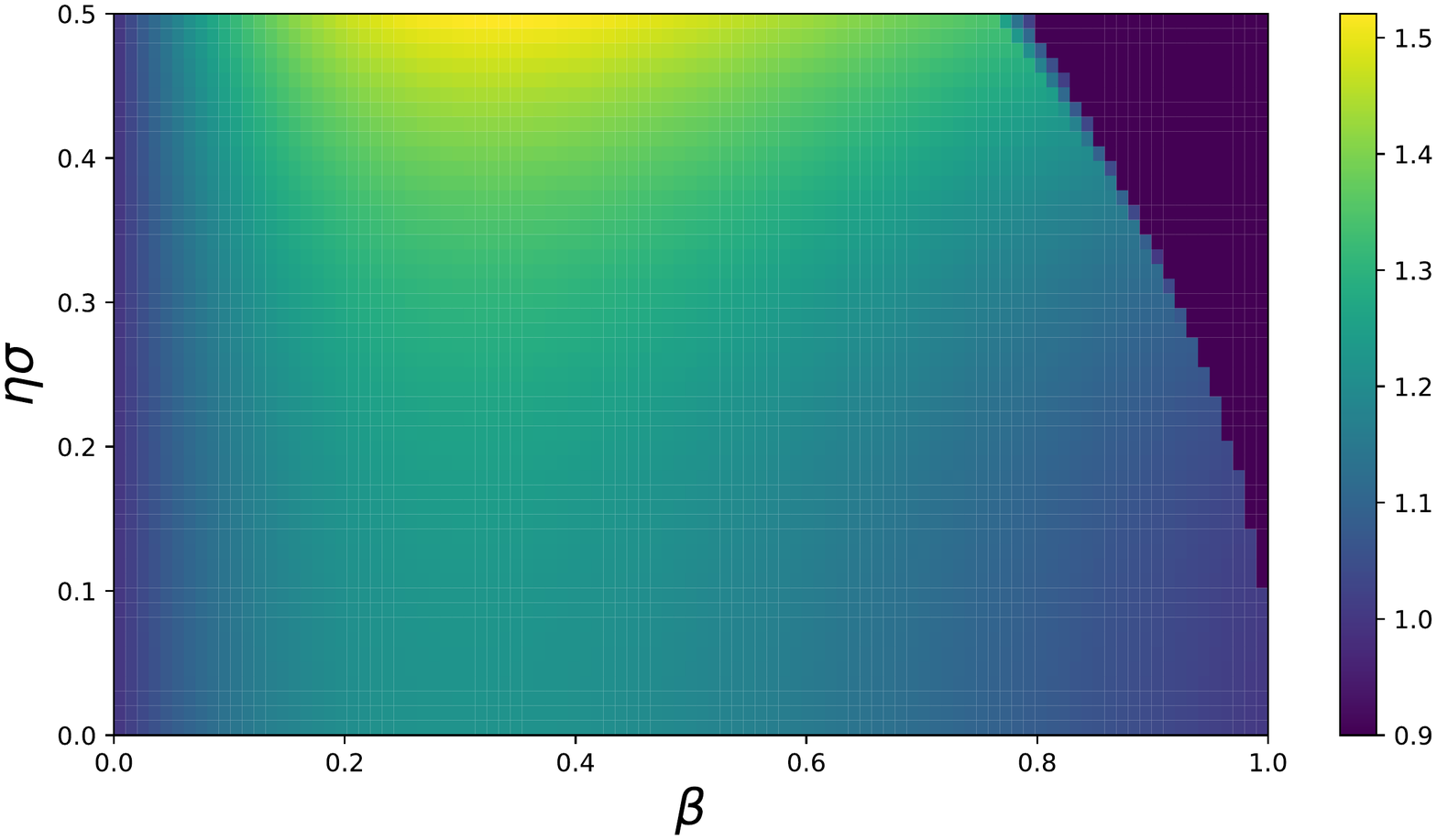}}
\caption{$\eta_1=\eta_2$, $\beta_1=0$, $\beta=-\beta_2$ (negative momentum) is the $x$-axis, $\eta\sigma_i$ is the $y$-axis. The optimal value of $\beta_2$ is always very close to $-0.3$. The dark area is where the method diverges.}
\end{subfigure}
\caption{Values of the spectral radius of the extragradient momentum matrix~\eqref{eq:momentum_matrix} for bilinear problems for different values of $\eta\sigma$ and $\beta$. The heat values is the multiplicative speed up from using $\beta>0$ compared to $\beta=0$, which we define as the ratio $\frac{\rho(\mT(\eta\sigma, \beta))}{\rho(\mT(\eta\sigma, 0))}$, where $\rho(\mA)$ is the spectral radius of a matrix $\mA$ for any $\mA$ and $\mT(\eta\sigma, \beta)$ is the value of matrix in the update under given $\eta\sigma$ and $\beta$, see~\eqref{eq:momentum_matrix} in Appendix~\ref{sec:momentum}.}
\label{fig:bilinear_heatmap}
\end{figure*}

\subsection{Discussion}
The bilinear example is very clear and the results that we obtained showed enough stability. However, the message from training GANs is very vague due to their well-known instability. We did not observe a significant impact of negative momentum on convergence speed or stability, but at the same time we mentioned that setting first momentum to 0 in Adam is important for the extra update to have impact. We believe that the bilinear problem in this situation is the best way to make conclusion, but we still aim to obtain new methods for GANs in future.

It is also worth mentioning that the actual loss functions used in GANs are typically nonsmooth due to the choice of loss functions. For instance, the popular WGAN formulation~\citep{arjovsky2017wasserstein} includes hinge loss. On top of that, neural networks themselves have nonsmooth activations such as ReLU and its variants. Therefore, it is an interesting direction to understand what happens when the assumptions typical to variational inequalities are violated.

\newpage

\bibliographystyle{plainnat}
\bibliography{extra_vi.bib}
\appendix
\clearpage
\onecolumn
\part*{Appendix: \Large ``Revisiting Stochastic Extragradient''}
\section{Proofs}
\textbf{Proof of Theorem~\ref{th:approx}}\\
We prove a more general version of the claim made in the main part, in particular we provide $O(\eta^k)$ bound for extragraident with $k$ steps. The precise claim is given below.
\begin{theorem}
	Let $F$ be an $L$-Lipschitz operator and define recursively $y_0 = x$ and $y_{m+1}\eqdef \prox{\eta g}{x - \eta F(y_m)}$ for $m=1,\dotsc, k$ and let $w\eqdef \prox{\eta g}{x - \eta F(w)}$ be the implicit update, where $\eta>0$ is any stepsize. Then,
	\begin{align*}
		\|w - y_k\|
		\le \eta^k L^k \|w - x\|.
	\end{align*}
\end{theorem}
\begin{proof}
	We show the claim by induction. For $k=0$ it holds simply because $y_0\eqdef x$. If it holds for $k - 1$, let us show it for $k$. By non-expansiveness of the proximal operator we have
	\begin{align*}
		\|w - y_{k}\|
		&= \|\prox{\eta g}{x - \eta F(w)} - \prox{\eta g}{x - \eta F(y_{k-1})}\| \\
		&\le \|x - \eta F(w) - (x - \eta F(y_{k-1}))\| \\
		&= \eta \|F(w) - F(y_{k-1})\| \\
		&\le \eta L \|w - y_{k-1}\| \\
		&\le \eta^{k} L^{k} \|w - x\|.
	\end{align*}
\end{proof}

\textbf{Proof of Theorem~\ref{th:linear_rate_str_mon}}

First, let us introduce the following lemma that will be very useful in our analysis.
\begin{lemma}\label{lem:prox-strong}
  Let $g$ be $\mu$--strongly convex and  $z = \prox{\eta g}{x}$. Then for all $y \in
  \R^d$ the following inequality holds:
  \begin{equation}
    \<z - x, y-z> \geq \eta \bigl(g(z)-g(y) + \frac{\mu}{2}\|z-y\|^2\bigr).
  \end{equation}
\end{lemma}
\begin{proof}
	The lemma easily follows from the definitions. Indeed, since
	\begin{align*}
		z
		\eqdef \arg\min_u \{\eta g(u) + \frac{1}{2}\|u - x\|^2\},
	\end{align*}
	we have necessary optimality condition
        $0 \in \eta \partial g(z) + (z - x)$. Thus, by the definition
        of a subdifferential and by strong convexity,
        \[\eta (g(y)-g(z))\geq \<x-z,y-z> + \frac{\eta
                \mu}{2}\|z-y\|^2\]
        and the proof is complete.
\end{proof}
In addition, let us also separately state how we are going to deal with the update variance.
\begin{lemma}\label{lem:eg_variance}
	Let $F(\cdot; \xi)$ be almost surely monotone and assume that point $x$ is such that $\sigma_x^2 \eqdef \EE \|F(x; \xi)- F(x)\|^2 < +\infty$, i.e., the variance of $F$ at $x$ is bounded. Then,
	\begin{align*}
		\EE \<F(x) - F(x; \xi^t), y^t - x>
		\le \eta\sigma_x^2 +  \frac{1}{4\eta}\EE\|y^t - x^t\|^2.
	\end{align*}
\end{lemma}
\begin{proof}
	As $x^t$ and $\xi^t$ are independent random variables and $\EE F(x; \xi^t) = F(x)$, we have
\begin{align*}
	\EE \<F(x) - F(x; \xi^t), y^t - x>
	&= \EE\<F(x) - F(x; \xi^t), x^t - x> + \EE\<F(x) - F(x; \xi^t), y^t - x^t> \\
	&= \EE\<F(x) - F(x; \xi^t), y^t - x^t>.
\end{align*}
By Young's inequality,
\begin{align*}
	\EE\<F(x) - F(x; \xi^t), y^t - x^t>
	&\le \eta\EE\|F(x) - F(x; \xi^t)\|^2 + \frac{1}{4\eta}\EE\|y^t - x^t\|^2 \\
	&= \eta\sigma_x^2 +  \frac{1}{4\eta}\EE\|y^t - x^t\|^2 
\end{align*}
and the proof is complete.
\end{proof}
Now we are ready to prove Theorem~\ref{th:linear_rate_str_mon}.
\begin{proof}
	By Lemma~\ref{lem:prox-strong} for points $y^t = \prox{\eta g}{x^t - \eta F(x^t; \xi^t)}$ and $x^{t+1} = \prox{\eta g}{x^t - \eta F(y^t; \xi^t)}$,
  \begin{align*}
    \<x^{t+1}-x^t + \eta F(y^t; \xi^t),x^*-x^{t+1}> 
    & \geq \eta \bigl(
    g(x^{t+1})-g(x^*) + \frac{\mu}{2} 
    \n{x^{t+1}-x^*}^2\bigr)  \\
    \<y^t-x^t + \eta F(x^t; \xi^t),x^{t+1}-y^t> 
    & \geq \eta \bigl(
    g(y^t)-g(x^{t+1}) + \frac{\mu}{2}
    \|x^{t+1}-y^t\|^2\bigr). 
  \end{align*}
  Summing these two inequalities together and rearranging, we get
  \begin{multline*}
    \lr{x^{t+1}-x^t,x^*-x^{t+1}} + \lr{y^t-x^t,
    x^{t+1}-y^t} + \eta \lr{F(y^t; \xi^t)-F(x^t; \xi^t), y^t-x^{t+1}} + \eta
    \lr{F(y^t; \xi^t),x^*-y^t} \\ \geq   \eta \bigl(g(y^t)  - g(x^*) + \frac
    \mu 2 \n{x^{t+1}-x^*}^2 + \frac \mu 2 \n{x^{t+1}-y^t}^2 \bigr).
  \end{multline*}
  Using identity  $2\lr{a,b} = \n{a+b}^2 - \n{a}^2-\n{b}^2$ for the
  first two scalar products, we deduce
  \begin{align*}
    (1+\eta \mu) \n{x^{t+1}-x^*}^2&\leq \n{x^t-x^*}^2  -
                                                  \n{x^t-y^t}^2-(1+\eta\mu)\n{x^{t+1}-y^t}^2\\
                                                  &+2\eta \lr{F(y^t; \xi^t)-F(x^t; \xi^t), y^t-x^{t+1}}-2\eta\bigl(\lr{F(y^t; \xi^t),y^t-x^*} +g(y^t)-g(x^*)\bigr).
  \end{align*}
  The first scalar product can be simplified using Lipschitzness. Since $F(\cdot; \xi^t)$ is almost surely $L$--Lipschitz, by Young's inequality
\begin{align*}
  2\eta \lr{F(y^t; \xi^t)-F(x^t; \xi^t), y^t-x^{t+1}} 
  &\le \frac{\eta}{L}\|F(y^t; \xi^t)-F(x^t; \xi^t)\|^2 + \eta L\|y^t-x^{t+1}\|^2 \\
  &\leq \eta L \bigl(\n{x^{t+1}-y^t}^2 + \n{y^t-x^t}^2\bigr).
\end{align*}
To get rid of the other scalar product, we use monotonicity of $F(\cdot; \xi^t)$, and then apply strong convexity of $g$,
\begin{align*}
    \label{extra_lin:mon}
     \lr{F(y^t; \xi^t),y^t-x^*} + g(y^t)-g(x^*)
     &\geq \lr{F(x^*; \xi^t),y^t-x^*} + g(y^t)-g(x^*) \\
     &=\lr{F(x^*),y^t-x^*} + g(y^t)-g(x^*) + \<F(x^*; \xi^t) - F(x^*), y^t - x^*> \\
     &\geq \frac{\mu}{2}\n{y^t-x^*}^2+ \<F(x^*; \xi^t) - F(x^*), y^t - x^*>.
\end{align*}
So far, the proof has not involved any expectation, but now we shall use Lemma~\ref{lem:eg_variance} to deduce from the produced bounds
\begin{align*}
	(1+\eta \mu) \EE\n{x^{t+1}-x^*}^2
	&\leq \EE\Bigl[\n{x^t-x^*}^2 - \eta \mu
    \bigl(\n{y^t-x^*}^2 + \n{x^{t+1}-y^t}^2\bigr)\Bigr] + 2\eta^2\sigma^2\\
    &\quad - \underbrace{(1 - \eta L - \tfrac{1}{2})}_{\ge 0}\EE\|y^t - x^t\|^2  \\
    &\le \EE\Bigl[\n{x^t-x^*}^2 - \eta \mu
    \bigl(\n{y^t-x^*}^2 + \n{x^{t+1}-y^t}^2\bigr)\Bigr]+ 2\eta^2\sigma^2.
\end{align*}
  Using inequality $ \n{a}^2+\n{b}^2\geq \frac 1 2 \n{a+b}^2$, we arrive at
  \begin{equation*}
    \Bigl(1+\frac 3 2 \eta \mu \Bigr) \n{x^{t+1}-x^*}^2\leq \n{x^t-x^*}^2 + 2\eta^2\sigma^2.
  \end{equation*}
  Note that $\eta\mu\le 1/2$ and, therefore, $\frac{1}{1 +
      3\eta\mu/2}\le (1 - 2\eta\mu/3)$. The statement of the theorem
  can be now easily obtained by induction.
\end{proof}

\textbf{Proof of Theorem~\ref{th:ergodic_conv}}

Let  $x\in \cX$. Similarly to the proof of Theorem~\ref{th:linear_rate_str_mon}, we can obtain from Lemma~\ref{lem:prox-strong} with $\mu=0$
  \begin{align*}
    \n{x^{t+1}-x}^2
    &\leq \n{x^t-x}^2  -\n{x^t-y^t}^2-\n{x^{t+1}-y^t}^2 + 2\eta L(\|x^{t+1} - y^t\|^2 + \|y^t - x^t\|^2\|)\\ 
    &\quad -2\eta\bigl(\lr{F(y^t; \xi^t),y^t-x} +g(y^t)-g(x)\bigr) \\
    &\le \n{x^t-x}^2  -\frac{1}{2}\n{x^t-y^t}^2-\n{x^{t+1}-y^t}^2 \\ 
    &\quad -2\eta\bigl(\lr{F(y^t; \xi^t),y^t-x} +g(y^t)-g(x)\bigr) .
  \end{align*}
By monotonicity of $F(\cdot; \xi^t)$ and Lemma~\ref{lem:eg_variance} we deduce
  \begin{align*}
  	\EE\<F(y^t; \xi^t), x - y^t>
  	&\le \EE\<F(x; \xi^t), x - y^t> \\
  	&\le \eta\sigma_x^2 + \EE\<F(x), x - y^t> + \frac{1}{4\eta}\EE\|y^t - x^t\|^2.
  \end{align*}
  Therefore,
  \begin{align*}
  	\EE\left[g(y^t) - g(x) + \<F(x), y^t - x>\right]
  	\le \frac{1}{2\eta}\EE\left[\|x^t-x\|^2 - \|x^{t+1} - x\|^2 \right] + \eta\sigma_x^2.
  \end{align*}
  Telescoping this inequality, we obtain
	\begin{align*}
		\EE\frac{1}{t+1}\sum_{k=0}^t(g(y^k) - g(x) + \< F(x), y^k - x>)
		\le \frac{1}{2\eta t}\|x^0 - x\|^2+ \eta \sigma_x^2
		\le \sup_{z\in\cX} \left\{\frac{1}{2\eta t}\|x^0 - z\|^2+ \eta \sigma_z^2\right\}.
	\end{align*}	  
  The left-hand side is a convex function $y^k$. Therefore, choosing $\eta = \cO\left( \frac{1}{\sqrt{t}} \right)$ and applying Jensen's inequality to the left-hand side, we get the  claim.

\textbf{Proof of Theorem~\ref{th:eg_bilinear}}
\begin{proof}
    Since the function is bilinear, we can write
    \begin{align*}
        \nabla_x f(x, y) = \mB(y-y^*), && \nabla_y f(x, y) = \mB^\top (x-x^*).
    \end{align*}
    Then, we obtain the explicit update rules
    \begin{align*}
        x^{t+1} 
        &= x^t - \eta_2 \mB(v^t - y^*)
        = x^t - \eta_2 \mB(y^t - y^* + \eta_1 \mB^\top (x^t - x^*))\\
        y^{t+1} 
        &= y^t + \eta_2 \mB^\top (u^t - x^*)
        = y^t + \eta_2 \mB^\top (x^t - x^* - \eta_1 \mB (y^t - y^*)).
    \end{align*}
    In matrix forms it is
    \begin{align*}
        \begin{bmatrix}x^{t+1} - x^* \\ y^{t+1} - y^* \end{bmatrix}
        = \begin{pmatrix}\mI - \eta_1\eta_2 \mB\mB^\top  & -\eta_2 \mB \\ \eta_2\mB^\top & \mI - \eta_1\eta \mB^\top\mB\end{pmatrix}
        \begin{bmatrix}x^{t} - x^* \\ y^{t} - y^* \end{bmatrix}
    \end{align*}
    Apply SVD decomposition to $\mB$: $\mB = \mU\mSigma \mV^\top$, where $\mU$ and $\mV$ are orthogonal and $\mSigma=\mathrm{diag}(\sigma_1, \dotsc, \sigma_n)$. Then,
    \begin{align*}
        \left\|\begin{bmatrix}x^{t+1} - x^* \\ y^{t+1} - y^* \end{bmatrix} \right\|
        \le \left\|\begin{pmatrix}\mI - \eta_1\eta_2 \mB\mB^\top  & -\eta_2 \mB \\ \eta_2\mB^\top & \mI - \eta_1\eta_2 \mB^\top\mB\end{pmatrix} \right\|
        \left\|\begin{bmatrix}x^{t} - x^* \\ y^{t} - y^* \end{bmatrix} \right\|.
    \end{align*}
    Since $\mU$ and $\mV$ are orthogonal, we have
    \begin{align*}
        \mB\mB^\top 
        &= \mU \mSigma^2\mV^\top,\\
        \mB^\top\mB
        &= \mV \mSigma^2 \mU^\top,
    \end{align*}
    and
    \begin{align*}
        \left\|\begin{pmatrix}\mI - \eta_1\eta \mB\mB^\top  & -\eta_2 \mB \\ \eta_2\mB^\top & \mI - \eta_1\eta \mB^\top\mB\end{pmatrix} \right\|
        &= \left\|\begin{pmatrix} \mU& 0\\0 &\mV \end{pmatrix}\begin{pmatrix}\mI - \eta_1\eta \mSigma^2  & -\eta_2 \mSigma \\ \eta_2\mSigma & \mI - \eta_1\eta \mSigma^2\end{pmatrix} \begin{pmatrix} \mU^\top& 0\\0 &\mV^\top \end{pmatrix}\right\|\\
        &= \left\|\begin{pmatrix}\mI - \eta_1\eta_2 \mSigma^2  & -\eta_2 \mSigma \\ \eta_2\mSigma & \mI - \eta_1\eta_2 \mSigma^2\end{pmatrix}\right\| \\
        &= \max_i \left\|\begin{pmatrix}1 - \eta_1\eta_2 \sigma_i^2  & -\eta_2 \sigma_i \\ \eta_2\sigma_i & 1 - \eta_1\eta_2 \sigma_i^2\end{pmatrix}\right\|\\
        &= \max_i\sqrt{(1 - \eta_1\eta_2\sigma_i^2)^2 + \eta_2^2\sigma_i^2}.
    \end{align*}
    Assume without loss of generality that
    $\sigma_1\ge\dotsb\ge\sigma_n$. Note that function $x\mapsto \left(1-\frac{\eta_1}{\eta_2}x^2\right)^2 + x^2$ is monotonically decreasing on $\left(0, c\right)$ and monotonically increasing on $\left(c, +\infty\right)$, where $c$ is $+\infty$ if $\eta_2\ge 2\eta_1$ and $ \frac{\eta_2}{\sqrt{2}\eta_1}\sqrt{2\frac{\eta_1}{\eta_2}-1}$ otherwise. Consequently, it holds
    \begin{align*}
        \max_i\{(1 - \eta_1\eta_2\sigma_i^2)^2 + \eta_2^2\sigma_i^2\}
        &= \max\{(1 - \eta_1\eta_2\sigma_1^2)^2 + \eta_2^2\sigma_1^2, (1 - \eta_1\eta_2\sigma_n^2)^2 + \eta_2^2\sigma_n^2\}.
    \end{align*}
\end{proof}
\textbf{Proof of Corollary~\ref{cor:bilinear}}
\begin{proof}
    These statements follow from the bound obtained in Theorem~\ref{th:eg_bilinear}. Since function $(1 - x^2)^2 + x^2$ monotonically decreases when $x\in\left(0, \frac{1}{\sqrt{2}}\right)$, we have $\rho=(1 - \eta_1\eta_2\sigma_{\min}(\mB)^2)^2 + \eta_2^2\sigma_{\min}(\mB)^2 = \left(1 - \frac{\sigma_{\min}(\mB)^2}{2\sigma_{\max}(\mB)^2}\right)^2 + \frac{\sigma_{\min}(\mB)^2}{2\sigma_{\max}(\mB)^2}$. The second case follows similarly.
\end{proof}
\subsection{Negative momentum}\label{sec:momentum}
For bilinear problems with two types of momentum the update recurrence is
\begin{align*}
        \begin{bmatrix}x^{t+1} - x^* \\ y^{t+1} - y^* \\ x^t - x^* \\ y^t - y^* \end{bmatrix}
        = \begin{pmatrix}
             (1 + \beta_2)\mI - \eta_1\eta_2 \mB\mB^\top  & -\eta_2(1+\beta_1) \mB & -\beta_2\mI & \eta_2\beta_1\mB \\
             \eta_2(1 + \beta_1)\mB^\top & (1 + \beta_2)\mI - \eta_1\eta \mB^\top\mB & -\eta_2\beta_1\mI & -\beta_2\mI  \\
             \mI & 0 & 0 & 0 \\
             0 & \mI & 0 & 0
        \end{pmatrix}
        \begin{bmatrix}x^{t} - x^* \\ y^{t} - y^*\\ x^{t-1} - x^* \\ y^{t-1} - y^* \end{bmatrix}.
\end{align*}
Using SVD decomposition, we can represent the above matrix as block-diagonal with blocks $\mT_i$
\begin{align} \label{eq:momentum_matrix}
	\mT_i	
	=\begin{pmatrix}
             1 + \beta_2 - \eta_1\eta_2 \sigma_i^2  & -\eta_2(1+\beta_1) \sigma_i & -\beta_2 & \eta_2\beta_1 \sigma_i \\
             \eta_2(1 + \beta_1)\sigma_i & 1 + \beta_2 - \eta_1\eta \sigma_i^2 & -\eta_2\beta_1 & -\beta_2  \\
             1 & 0 & 0 & 0 \\
             0 & 1 & 0 & 0
        \end{pmatrix},
\end{align}
where $\sigma_i$ is the $i$-th the singular value of $\mB$.

One can show that the spectral radius of this matrix improves with negative $\beta_2$, however this is not true for its second norm. Since this is a very technical property that can be easily illustrated numerically, we simply provided a plot of how spectral radius changes depending on values of $\eta\sigma$ and $\beta_2$ when $\beta=1=0$ and $\eta_1=\eta_2=\eta$, see Figure~\ref{fig:bilinear_heatmap}. In addition, here we provide the heatmap for $\eta_1=\eta_2$ and product $\eta\sigma=0.01$. As can be seen from Figure~\ref{fig:bilinear_heatmap2}, nonzero $\beta_1$ is not very promising and $\beta_2$ leads only  to a small improvement. Thus, it gives advantage mainly for large values of $\eta\sigma$.
\begin{figure}
	\centering
	\includegraphics[scale=0.4]{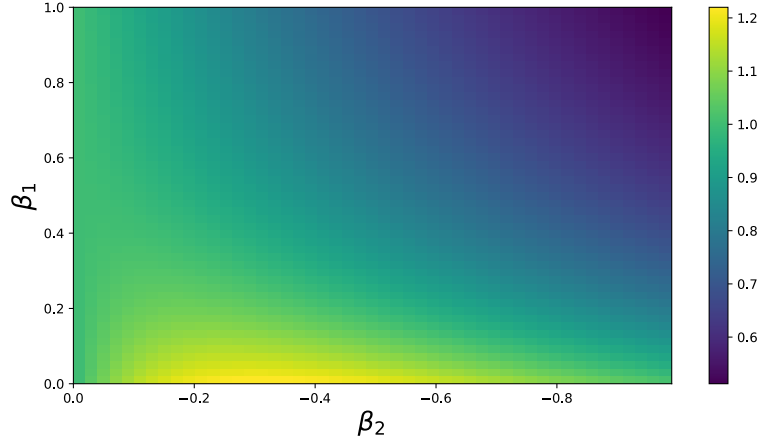}
	\caption{Ratio of spectral radii as in Figure~\ref{fig:bilinear_heatmap} but with fixed $\eta\sigma=0.01$ and different values of $\beta_1$ and $\beta_2$.}
	\label{fig:bilinear_heatmap2}
\end{figure}

\subsection{Proof of Theorem~\ref{th:nonconvex}}
Let us introduce a notation that simplifies the proof. We will denote by $\EE_t$ the expectation conditioned on $x^t$, i.e., $\EE_t [\cdot] = \EE [\cdot \mid x^t]$.
\begin{proof}
	Recall that $y^t = x^t - \eta \nabla f(x^t; \xi^t)$, $x^{t+1} = x^t - \eta \nabla f(y^t; \xi^t)$, and apply smoothness of $f$ to $x^{t+1}$ and $x^t$:
	\begin{align*}
		f(x^{t+1})
		&\le f(x^t) + \<\nabla f(x^t), x^{t+1} - x^t> + \frac{L}{2} \|x^{t+1} - x^t\|^2 \\
		&= f(x^t) - \eta\|\nabla f(x^t)\|^2 + \eta\<\nabla f(x^t), \nabla f(x^t) - \nabla f(y^t; \xi^t)> + \frac{L\eta^2}{2}\|\nabla f(y^t; \xi^t)\|^2.
	\end{align*}
	Since $\nabla f(x^t; \xi^t)$ is an unbiased  estimate of
        $\nabla f(x^t)$, it follows by Young's inequality and
        smoothness of $f(\cdot; \xi^t)$
	\begin{align*}
		\eta \< \nabla f(x^t), \nabla f(x^t) - \nabla f(y^t; \xi^t)>
		&= \EE_t\eta \< \nabla f(x^t), \nabla f(x^t; \xi^t) - \nabla f(y^t; \xi^t)>\\
		&\le \frac{\eta^2 L}{2}\|\nabla f(x^t)\|^2 + \frac{1}{2L}\EE_t \|\nabla f(x^t; \xi^t) - \nabla f(y^t; \xi^t)\|^2 \\
		&\le \frac{\eta^2L}{2}\|\nabla f(x^t)\|^2 + \frac{L}{2} \EE_t\|x^t - y^t\|^2 \\
		&= \frac{\eta^2L}{2}\|\nabla f(x^t)\|^2 + \frac{\eta^2 L}{2} \EE_t\|\nabla f(y^t; \xi^t)\|^2 .
	\end{align*}
	Moreover, similar arguments show how to bound the expectation of the squared gradient norm:
	\begin{align*}
		\EE_t \|\nabla f(y^t; \xi^t)\|^2
		&\le 2\EE_t \|\nabla f(y^t; \xi^t) - \nabla f(x^t; \xi^t)\|^2 + 2\EE_t \|\nabla f(x^t; \xi^t)\|^2 \\
		&\le 2L^2\EE_t\|y^t - x^t\|^2 + 2\EE_t \|\nabla f(x^t; \xi^t)\|^2 \\
		&= 2(1 + L^2 \eta^2) \EE_t\|\nabla f(x^t; \xi^t)\|^2 \\
		&\le 2(1 + L^2 \eta^2) (\|\nabla f(x^t)\|^2 + \sigma^2).
	\end{align*}
	Thus,
	\begin{align*}
		\EE_t f(x^{t+1})
		\le f(x^t) - \eta\left[1 - \eta L - 2\eta L(1 + \eta^2 L^2)\right]\|\nabla f(x^t)\|^2 + 2\eta^2 L(1 + \eta^2 L^2)\sigma^2.
	\end{align*}
	If $\eta L \le \frac{1}{4}$, we have $1 - \eta L - 2\eta L(1 + \eta^2 L^2) > \frac{1}{5}$, so this bound can be simplified to
	\begin{align*}
		\|\nabla f(x^t)\|^2
		\le \frac{5}{\eta}\EE_t[f(x^{t}) - f(x^{t+1})] + 11\eta L\sigma^2.
	\end{align*}
	Telescoping this inequality from $0$ to $t-1$ and taking full expectation with respect to all randomness, we get
	\begin{align*}
		\frac{1}{t}\sum_{k=0}^{t-1} \EE\|\nabla f(x^k)\|^2
		&\le \frac{5}{\eta t}(f(x^{0}) - f(x^{t})) + 11\eta L\sigma^2 \\
		&\le \frac{5}{\eta t}(f(x^{0}) - f^*) + 11\eta L\sigma^2 .
	\end{align*}
	It remains to mention that the left-hand side is exactly the expectation of $\EE\|\nabla f(\hat x^t)\|^2$.
\end{proof}

\section{Additional experiments}
\subsection{Reproducing mixture of eight Gaussians}
We also double check that extragradient converges on the mixture of 8 Gaussians. This experiment is a sanity that allows us to show that the method can do at least as well as alternating gradient~\citep{gidel2018negative}. To directly relate to their experiments, we ran extragradient on the same type of network, although we changed activation from ReLU to tanh, which was more stable in our experiments. Note that~\citep{gidel2018negative} ran alternating method for 100,000 iterations, while we required only 20,000, which corresponds to 40,000 generator updates. The result is presented in Figure~\ref{fig:8gaussian}.
\begin{figure}
\centering
{\includegraphics[scale=0.4]{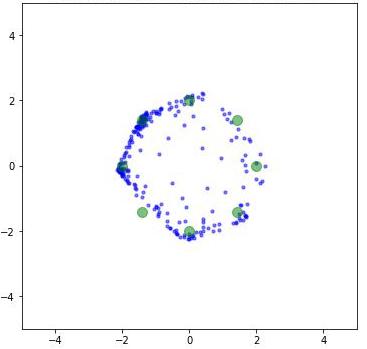}}
\caption{Samples from generator after training for 20,000 iterations of minibatch 512 with extragradient. Both generator and discriminator are 4-layers neural networks with tanh activation and the dimension of the noise distribution is 256.}
\label{fig:8gaussian}
\end{figure}
\subsection{Empirical risk minimization}\label{sec:erm_exp}
As our theory suggests, stochastic extragradient might not be better than SGD when solving a simple task such as function minimization. To see how it works in practice, we trained Residual Network~\citep{he2016deep}, Resnet-18, on Cifar10~\citep{krizhevsky2009learning} dataset with cross-entropy loss and different stepsizes, and compared the results to SGD. In order to see the effect of the update rule, we do not use any type of momentum in this experiment and keep the learning rate constant. Our observation in this situation is that extragradient is indeed slower, both because of the need to compute two gradients per iterations and because of worse final accuracy.
\begin{figure}
\centering
{\includegraphics[scale=0.25]{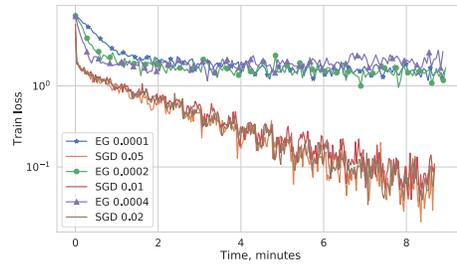}}
\caption{Comparison of the proposed stochastic extragradient and stochastic gradient descent when optimizing Residual Network with 18 hidden layers on Cifar10 dataset. We report only the train loss as  this is the most relevant metric for an optimization method, and test accuracy in this experiment behaved similarly.}
\end{figure}

\subsection{Samples of generated images}
\begin{figure}[bh]
	\centering
	\includegraphics[scale=0.3]{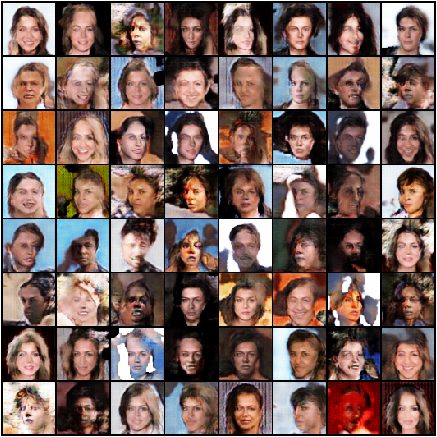}
	\includegraphics[scale=0.3]{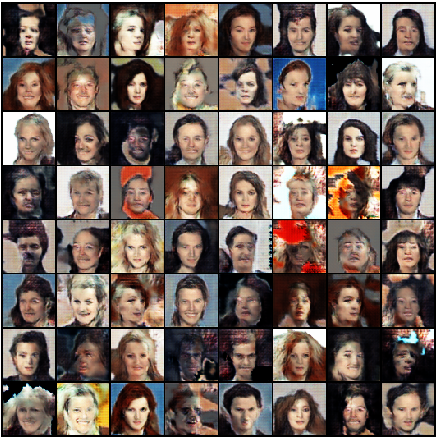}
	\includegraphics[scale=0.3]{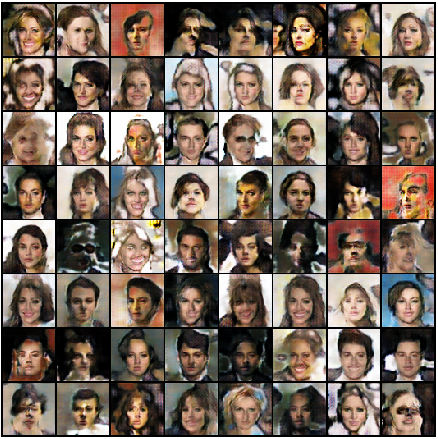}\\
	\includegraphics[scale=0.3]{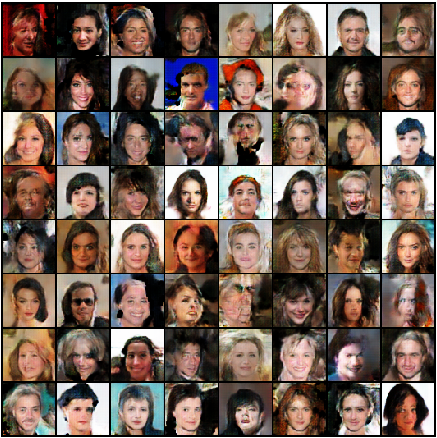}
	\includegraphics[scale=0.3]{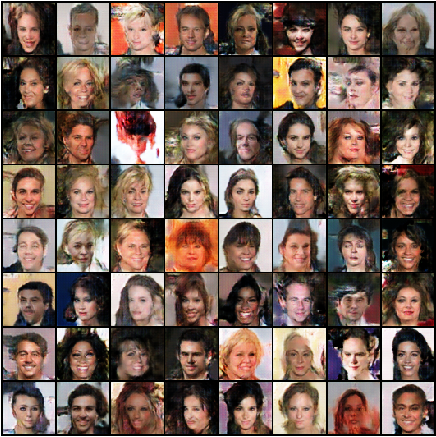}
	\includegraphics[scale=0.3]{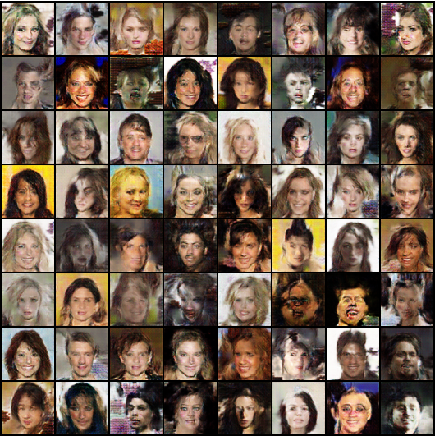}
	\caption{Adam (top) and ExtraAdam (bottom) results of training self attention GAN for two epochs. The results of training with the three best performing stepsizes, $10^{-3}, 2\cdot 10^{-3}, 4\cdot 10^{-3}$, are provided for each method (from the left to the right). Best seen in color by zooming on a computer screen.}\label{fig:celeba}
\end{figure}

\end{document}